\documentclass[11pt, notitlepage]{article}   

\usepackage{amsmath,amsthm,amsfonts}   

\usepackage{amscd}
\usepackage{amsfonts}
\usepackage{amssymb}
\usepackage{color}      
\usepackage{epsfig}
\usepackage{graphicx}           

\usepackage{tikz}

\theoremstyle{plain}
\newtheorem{theorem}{Theorem}[section]
\newtheorem{lemma}[theorem]{Lemma}
\newtheorem{corollary}[theorem]{Corollary}
\newtheorem{proposition}[theorem]{Proposition}
\newtheorem{observation}[theorem]{Observation}
\newtheorem{example}[theorem]{Example}
\newtheorem{remark}[theorem]{Remark}

\newtheorem{question}[theorem]{Question}

\theoremstyle{definition}

\newcommand{\adim}{\textnormal{adim}}
\newcommand{\diam}{\textnormal{diam}}
\newcommand{\code}{\textnormal{code}}

\def\finf{\mathop{{\rm I}\kern -.27 em {\rm F}}\nolimits}

\date{}

\begin{document}

\title{Distance-$k$ locating-dominating sets in graphs}

\author{{\bf{Cong X. Kang}}$^1$ and {\bf{Eunjeong Yi}}$^2$\\
\small Texas A\&M University at Galveston, Galveston, TX 77553, USA\\
{\small\em kangc@tamug.edu}$^1$; {\small\em yie@tamug.edu}$^2$}

\maketitle

\begin{abstract}
Let $G$ be a graph with vertex set $V$, and let $k$ be a positive integer. A set $D \subseteq V$ is a \emph{distance-$k$ dominating set} of $G$ if, for each vertex $u \in V-D$, there exists a vertex $w\in D$ such that $d(u,w) \le k$, where $d(u,w)$ is the minimum number of edges linking $u$ and $w$ in $G$. Let $d_k(x, y)=\min\{d(x,y), k+1\}$. A set $R\subseteq V$ is a \emph{distance-$k$ resolving set} of $G$ if, for any pair of distinct $x,y\in V$, there exists a vertex $z\in R$ such that $d_k(x,z) \neq d_k(y,z)$. The \emph{distance-$k$ domination number} $\gamma_k(G)$ (\emph{distance-$k$ dimension} $\dim_k(G)$, respectively) of $G$ is the minimum cardinality of all distance-$k$ dominating sets (distance-$k$ resolving sets, respectively) of $G$. The \emph{distance-$k$ location-domination number}, $\gamma_L^k(G)$, of $G$ is the minimum cardinality of all sets $S\subseteq V$ such that $S$ is both a distance-$k$ dominating set and a distance-$k$ resolving set of $G$. Note that $\gamma_L^1(G)$ is the well-known location-domination number introduced by Slater in 1988. For any connected graph $G$ of order $n\ge 2$, we obtain the following sharp bounds: (1) $\gamma_k(G) \le \dim_k(G)+1$; (2) $2\le\gamma_k(G)+\dim_k(G) \le n$; (3) $1\le \max\{\gamma_k(G), \dim_k(G)\} \le \gamma_L^k(G) \le \min\{\dim_k(G)+1, n-1\}$. We characterize $G$ for which $\gamma_L^k(G)\in\{1, |V|-1\}$. We observe that $\frac{\dim_k(G)}{\gamma_k(G)}$ can be arbitrarily large. Moreover, for any tree $T$ of order $n\ge 2$, we show that $\gamma_L^k(T)\le n-ex(T)$, where $ex(T)$ denotes the number of exterior major vertices of $T$, and we characterize trees $T$ achieving equality. We also examine the effect of edge deletion on the distance-$k$ location-domination number of graphs. 
\end{abstract}




\section{Introduction}

Let $G$ be a finite, simple, undirected, and connected graph with vertex set $V(G)$ and edge set $E(G)$. Let $k$ be a positive integer. For $x,y\in V(G)$, let $d(x, y)$ denote the length of a shortest path between $x$ and $y$ in $G$, and let $d_k(x,y)=\min\{d(x,y), k+1\}$. The \emph{diameter}, $\diam(G)$, of a graph $G$ is $\max\{d(x,y): x,y \in V(G)\}$. For $v\in V(G)$ and $S \subseteq V(G)$, let $d(v,S)=\min\{d(v,w):w\in S\}$. The \emph{open neighborhood} of a vertex $v \in V(G)$ is $N(v)=\{u \in V(G) : uv \in E(G)\}$ and its \emph{closed neighborhood} is $N[v]=N(v) \cup \{v\}$. More generally, for $v\in V(G)$, let $N^k[v]=\{u\in V(G): d(u, v)\le k\}$. The \emph{degree} of a vertex $v \in V(G)$ is $|N(v)|$. For distinct $x,y\in V(G)$, $x$ and $y$ are called \emph{twin vertices} if $N(x)-\{y\}=N(y)-\{x\}$ in $G$. A \emph{major vertex} is a vertex of degree at least three, a \emph{leaf} (also called an \emph{end-vertex}) is a vertex of degree one, and a \emph{support vertex} is a vertex that is adjacent to a leaf. A leaf $\ell$ is called a \emph{terminal vertex} of a major vertex $v$ if $d(\ell,v) < d(\ell,w)$ for every other major vertex $w$ in $G$. The terminal degree, $ter(v)$, of a major vertex $v$ is the number of terminal vertices of $v$ in $G$. A major vertex $v$ is an \emph{exterior major vertex} if it has positive terminal degree. We denote the number of exterior major vertices of $G$ by $ex(G)$ and the number of leaves of $G$ by $\sigma(G)$. We denote by $\overline{G}$ the complement of $G$, i.e., $V(\overline{G})=V(G)$ and $xy\in E(\overline{G})$ if and only if $xy \not\in E(G)$ for any distinct vertices $x$ and $y$ in $G$. The \emph{join} of two graphs $G$ and $H$, denoted by $G+H$, is the graph obtained from the disjoint union of $G$ and $H$ by joining an edge between each vertex of $G$ and each vertex of $H$. Let $P_n$, $C_n$, and $K_n$ denote respectively the path, the cycle, and the complete graph on $n$ vertices; let $K_{s, n-s}$ denote the complete bi-partite graph on $n$ vertices with parts of sizes $s$ and $n-s$. Let $\mathbb{Z}^+$ be the set of positive integers and $k\in\mathbb{Z}^+$. For $\alpha\in \mathbb{Z}^+$, let $[\alpha]=\{1,2,\ldots, \alpha\}$.

A vertex subset $D \subseteq V(G)$ is a \emph{distance-$k$ dominating set} of $G$ if, for each vertex $u \in V(G)-D$, there exists a vertex $w\in D$ such that $d(u,w) \le k$. The \emph{distance-$k$ domination number}, $\gamma_k(G)$, of $G$ is the minimum cardinality over all distance-$k$ dominating sets of $G$. The concept of distance-$k$ domination was introduced by Meir and Moon~\cite{distKdom}. We note that $\gamma_1(G)$ is the well-known domination number of $G$, which is often denoted by $\gamma(G)$ in the literature. Applications of domination can be found in resource allocation on a network, determining efficient routes within a network, and designing secure systems for electrical grids, to name a few. It is known that determining the domination number of a general graph is an NP-hard problem (see~\cite{NPcompleteness}). For a survey on domination in graphs, see \cite{Dombook2}. 

A vertex subset $R \subseteq V(G)$ is a \emph{resolving set} of $G$ if, for any pair of distinct vertices $x,y\in V(G)$, there exists a vertex $z\in R$ such that $d(x,z) \neq d(y,z)$. The \emph{metric dimension}, $\dim(G)$, of $G$ is the minimum cardinality over all resolving sets of $G$. The concept of metric dimension was introduced independently by Slater~\cite{Slater} and by Harary and Melter \cite{HM}. A vertex subset $S\subseteq V(G)$ is a \emph{distance-$k$ resolving set} (also called a \emph{$k$-truncated resolving set}) of $G$ if, for any distinct vertices $x,y\in V(G)$, there exists a vertex $z\in S$ such that $d_k(x,z) \neq d_k(y,z)$. The \emph{distance-$k$ dimension} (also called the \emph{$k$-truncated dimension}), $\dim_k(G)$, of $G$ is the minimum cardinality over all distance-$k$ resolving sets of $G$. The metric dimension of a metric space $(V, d_k)$ is studied in~\cite{beardon}. The distance-$k$ dimension corresponds to the $(1, k+1)$-metric dimension in~\cite{moreno1} and~\cite{moreno2}. We note that $\dim_1(G)$ is also called the adjacency dimension, introduced in~\cite{adim}, and it is often denoted by $\adim(G)$ in the literature. For detailed results on $\dim_k(G)$, we refer to~\cite{distKdim}, which is a merger of~\cite{JE} and~\cite{tilquist}, along with some additional results. For an ordered set $S=\{u_1, u_2, \ldots, u_{\alpha}\} \subseteq V(G)$ of distinct vertices, the distance-$k$ metric code of $v \in V(G)$ with respect to $S$, denoted by $\code_{S,k}(v)$, is the $\alpha$-vector $(d_k(v, u_1), d_k(v, u_2), \ldots, d_k(v, u_{\alpha}))$. We denote by $({\bf k+1})_{\alpha}$ the $\alpha$-vector with $k+1$ on each entry. Applications of metric dimension can be found in robot navigation, network discovery and verification, and combinatorial optimization, to name a few. It is known that determining the metric dimension and the adjacency dimension of a general graph are NP-hard problems (see~\cite{landmarks} and~\cite{Juan}). For a discussion on computational complexity of the distance-$k$ dimension of graphs, see~\cite{moreno2}.

Slater~\cite{loc_dom} introduced the notion of locating-dominating set and location-domination number. A set $A\subseteq V(G)$ is a \emph{locating-dominating set} of $G$ if $A$ is a dominating set of $G$ and $N(x) \cap A \neq N(y) \cap A$ for distinct vertices $x,y\in V(G)-A$. The \emph{location-domination number}, $\gamma_L(G)$, of $G$ is the minimum cardinality over all locating-dominating sets of $G$.  The notion of location-domination by Slater is a natural marriage of its two constituent notions, where a subset of vertices functions both to locate (via $d_1$ metric) each node of a network and to dominate (supply or support) the entire network.  Viewed in this light, the following is but a natural extension of the notion of Slater. For $(s,t)\in\mathbb{Z}^+\times\mathbb{Z}^+$, let $S\subseteq V(G)$ be a distance-$s$ resolving set of $G$ and a distance-$t$ dominating set of $G$, which we call an $(s,t)$-locating-dominating set of $G$. Then the $(s,t)$-location-domination number of $G$, denoted by $\gamma^{(s,t)}_L\!(G)$, is defined to be the minimum cardinality of $S$ as $S$ varies over all $(s,t)$-locating-dominating sets of $G$. When $s=k=t$, we will abbreviate and simply speak of distance-$k$ locating-dominating set and distance-$k$ location-domination number, and we will simplify $\gamma^{(k,k)}_L\!(G)$ to $\gamma_L^k(G)$.

In this paper, we study the distance-$k$ location-domination number of graphs. We examine the relationship among $\gamma_k(G)$, $\dim_k(G)$ and $\gamma_L^k(G)$. Let $G$ be a connected graph of order $n\ge2$, and let $k\in\mathbb{Z}^+$. In Section~\ref{sec_kdom_kdim}, we show that $\gamma_k(G) \le \dim_k(G)+1$ and that $\dim_k(G)-\gamma_k(G)$ can be arbitrarily large. We also show that $2\le \gamma_k(G)+\dim_k(G) \le n$, and we characterize $G$ satisfying $\gamma_k(G)+\dim_k(G)=2$. In Section~\ref{sec_distK_loc_dom}, we show that $1\le \max\{\gamma_k(G), \dim_k(G)\} \le \gamma_L^k(G) \le \min\{\dim_k(G)+1, n-1\}$, where the bounds are sharp. We also characterize $G$ satisfying $\gamma_L^k(G)$ equals $1$ and $n-1$, respectively. Moreover, for a non-trivial tree $T$, we show that $\gamma_L^k(T)\le n-ex(T)$ and we characterize trees $T$ achieving equality. In Section~\ref{sec_graphs}, we determine $\gamma_L^k(G)$ when $G$ is the Petersen graph, a complete multipartite graph, a cycle or a path. In Section~\ref{sec_edge}, we examine the effect of edge deletion on the distance-$k$ location-domination number of graphs.


\section{Relations between $\gamma_k(G)$ and $\dim_k(G)$}\label{sec_kdom_kdim}

In this section, we examine the sum and difference between $\gamma_k(G)$ and $\dim_k(G)$. Let $G$ be a non-trivial connected graph, and let $k\in\mathbb{Z}^+$. We show that $\gamma_k(G)\le \dim_k(G)+1$, where the bound is sharp, and we observe that $\dim_k(G)-\gamma_k(G)$ can be arbitrarily large. We also show that $2\le \gamma_k(G)+\dim_k(G) \le |V(G)|$, and we characterize $G$ satisfying $\gamma_k(G)+\dim_k(G) =2$. We begin with the following observation.

\begin{observation}\label{obs_mixed}
Let $G$ be any connected graph, and let $s,s',t,t',k,k'\in\mathbb{Z}^+$. Then 
\begin{itemize} 
\item[(a)] for $k>k'$, $\gamma_k(G)\le \gamma_{k'}(G)\le \gamma_1(G)$; 
\item[(b)] \emph{\cite{beardon, moreno1, moreno2}} for $k>k'$, $\dim(G) \le \dim_k(G)\le\dim_{k'}(G) \le \dim_1(G)$;
\item[(c)] more generally, we have $\gamma^{(s,t)}_L\!(G)\geq \gamma^{(s',t')}_L\!(G)$ for $s\leq s'$ and $t\leq t'$, since an $(s,t)$-locating-dominating set of $G$ is an $(s',t')$-locating-dominating set of $G$.
\end{itemize}
\end{observation}

For any minimum distance-$k$ resolving set $S$ of a connected graph $G$, we show that there is a vertex $v\in V(G)-S$ such that $S\cup\{v\}$ is a distance-$k$ dominating set of $G$.

\begin{proposition}\label{dom_upper}
For any non-trivial connected graph $G$ and for any $k\in\mathbb{Z}^+$, $$\gamma_k(G)\le \dim_k(G)+1.$$
\end{proposition}

\begin{proof}
Let $S$ be any minimum distance-$k$ resolving set of $G$. Then there exists at most one vertex, say $w$, in $V(G)-S$ such that $d(w,S)>k$; notice that $\code_{S,k}(w)=(\textbf{k+1})_{|S|}$. If $d(u,S)\le k$ for each $u\in V(G)$, then $S$ is a distance-$k$ dominating set of $G$, and hence $\gamma_k(G)\le |S|=\dim_k(G)$. If there exists a vertex $v\in V(G)$ such that $d(v,S)>k$, then $S \cup \{v\}$ forms a distance-$k$ dominating set of $G$, and thus $\gamma_k(G)\le |S|+1=\dim_k(G)+1$.~\hfill
\end{proof}

Next, we show the sharpness of the bound in Proposition~\ref{dom_upper}. 

\begin{observation}\label{obs_dom}
Let $G$ be a non-trivial connected graph. 
\begin{itemize}
\item[(a)] If there exists a vertex $v\in V(G)$ such that $N^k[v]=V(G)$, then $\{v\}$ is a distance-$k$ dominating set of $G$ and $\gamma_k(G)=1$.
\item[(b)] Suppose $\cup_{i=1}^{x}\{v_i\} \subseteq V(G)$ satisfies $N^k[v_i] \cap N^k[v_j]=\emptyset$ for $i\neq j$. Then any distance-$k$ dominating set of $G$ must contain a vertex of $N^k[v_i]$ for each $i\in[x]$. Thus $\gamma_k(G) \ge x$.
\end{itemize}
\end{observation}

\begin{remark}\label{remark_kdom_kdim}
For each $k\in\mathbb{Z}^+$, there is a connected graph $G$ with $\gamma_k(G)=\dim_k(G)+1$. 
\end{remark}

\begin{proof}
Let $G$ be a tree with $ex(G)=x\ge 1$ such that $v_1, v_2, \ldots, v_x$ are the exterior major vertices of $G$ with $ter(v_i)=\alpha\ge3$ for each $i\in[x]$, and let $v_1, v_2, \ldots, v_x$ form an induced path of order $x$ in $G$. For each $i\in[x]$, let $\{\ell_{i,1}, \ell_{i,2}, \ldots, \ell_{i, \alpha}\}$ be the set of the terminal vertices of $v_i$ in $G$ such that $d(v_i, \ell_{i,j})=k+1=1+d(v_i, \ell_{i,\alpha})$ for each $j\in[\alpha-1]$. For each $i\in[x]$ and for each $j\in[\alpha-1]$, let $s_{i,j}$ be the neighbor of $v_i$ lying on the $v_i-\ell_{i, j}$ path in $G$. See Fig.~\ref{fig_ex1} when $k=3$. 

First, we note that $\gamma_k(G)=x\alpha$: (i) $\gamma_k(G)\le x\alpha$ since $D=\cup_{i=1}^{x}\{\ell_{i,1}, \ldots, \ell_{i, \alpha}\}$ forms a distance-$k$ dominating set of $G$ with $|D|=x\alpha$; (ii) $\gamma_k(G)\ge x\alpha$ by Observation~\ref{obs_dom}(b) and the fact that $N^k[\ell_{i,j}]\cap N^k[\ell_{s,t}]=\emptyset$ for $(i,j)\neq (s,t)$. Second, we note that $\dim_k(G)=x\alpha-1$: (i) $\dim_k(G)\le x\alpha-1$ since $R=(\cup_{i=1}^{x}\{s_{i,1}, s_{i,2}, \ldots, s_{i, \alpha-1}\}) \cup (\cup_{i=1}^{x-1}\{\ell_{i, \alpha}\})$ forms a distance-$k$ resolving set of $G$ with $|R|=x\alpha-1$; (ii) $\dim_k(G) \ge \gamma_k(G)-1=x\alpha-1$ by Proposition~\ref{dom_upper}. Therefore, $\gamma_k(G)=x\alpha=\dim_k(G)+1$.~\hfill
\end{proof}

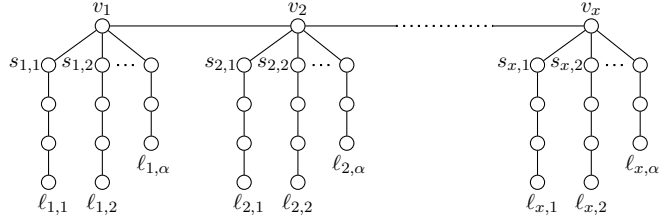
\begin{figure}[ht]
\centering
\begin{tikzpicture}[scale=.65, transform shape]

\node [draw, shape=circle, scale=.8] (1) at  (0.5,0) {};
\node [draw, shape=circle, scale=.8] (a11) at  (-0.6,-0.8) {};
\node [draw, shape=circle, scale=.8] (a12) at  (-0.6,-1.6) {};
\node [draw, shape=circle, scale=.8] (a13) at  (-0.6,-2.4) {};
\node [draw, shape=circle, scale=.8] (a14) at  (-0.6,-3.2) {};
\node [draw, shape=circle, scale=.8] (a21) at  (0.5,-0.8) {};
\node [draw, shape=circle, scale=.8] (a22) at  (0.5,-1.6) {};
\node [draw, shape=circle, scale=.8] (a23) at  (0.5,-2.4) {};
\node [draw, shape=circle, scale=.8] (a24) at  (0.5,-3.2) {};
\node [draw, shape=circle, scale=.8] (a31) at  (1.5,-0.8) {};
\node [draw, shape=circle, scale=.8] (a32) at  (1.5,-1.6) {};
\node [draw, shape=circle, scale=.8] (a33) at  (1.5,-2.4) {};

\draw(a14)--(a13)--(a12)--(a11)--(1)--(a21)--(a22)--(a23)--(a24);\draw(1)--(a31)--(a32)--(a33);\draw[thick, dotted](0.8,-0.8)--(1.2, -0.8);

\node [draw, shape=circle, scale=.8] (2) at  (4.5,0) {};
\node [draw, shape=circle, scale=.8] (b11) at  (3.4,-0.8) {};
\node [draw, shape=circle, scale=.8] (b12) at  (3.4,-1.6) {};
\node [draw, shape=circle, scale=.8] (b13) at  (3.4,-2.4) {};
\node [draw, shape=circle, scale=.8] (b14) at  (3.4,-3.2) {};
\node [draw, shape=circle, scale=.8] (b21) at  (4.5,-0.8) {};
\node [draw, shape=circle, scale=.8] (b22) at  (4.5,-1.6) {};
\node [draw, shape=circle, scale=.8] (b23) at  (4.5,-2.4) {};
\node [draw, shape=circle, scale=.8] (b24) at  (4.5,-3.2) {};
\node [draw, shape=circle, scale=.8] (b31) at  (5.5,-0.8) {};
\node [draw, shape=circle, scale=.8] (b32) at  (5.5,-1.6) {};
\node [draw, shape=circle, scale=.8] (b33) at  (5.5,-2.4) {};

\draw(b14)--(b13)--(b12)--(b11)--(2)--(b21)--(b22)--(b23)--(b24);\draw(2)--(b31)--(b32)--(b33);\draw[thick, dotted](4.8,-0.8)--(5.2, -0.8);

\node [draw, shape=circle, scale=.8] (3) at  (10.5,0) {};
\node [draw, shape=circle, scale=.8] (c11) at  (9.4,-0.8) {};
\node [draw, shape=circle, scale=.8] (c12) at  (9.4,-1.6) {};
\node [draw, shape=circle, scale=.8] (c13) at  (9.4,-2.4) {};
\node [draw, shape=circle, scale=.8] (c14) at  (9.4,-3.2) {};
\node [draw, shape=circle, scale=.8] (c21) at  (10.5,-0.8) {};
\node [draw, shape=circle, scale=.8] (c22) at  (10.5,-1.6) {};
\node [draw, shape=circle, scale=.8] (c23) at  (10.5,-2.4) {};
\node [draw, shape=circle, scale=.8] (c24) at  (10.5,-3.2) {};
\node [draw, shape=circle, scale=.8] (c31) at  (11.5,-0.8) {};
\node [draw, shape=circle, scale=.8] (c32) at  (11.5,-1.6) {};
\node [draw, shape=circle, scale=.8] (c33) at  (11.5,-2.4) {};

\draw(c14)--(c13)--(c12)--(c11)--(3)--(c21)--(c22)--(c23)--(c24);\draw(3)--(c31)--(c32)--(c33);\draw[thick, dotted](10.8,-0.8)--(11.2, -0.8);

\draw(1)--(2)--(6.5,0);\draw[thick, dotted](6.5,0)--(8.5, 0);\draw(8.5,0)--(3);

\node [scale=1.15] at (0.5,0.35) {$v_1$};
\node [scale=1.15] at (4.5,0.35) {$v_2$};
\node [scale=1.15] at (10.5,0.35) {$v_x$};

\node [scale=1.15] at (-1.08,-0.8) {$s_{1,1}$};
\node [scale=1.15] at (0,-0.8) {$s_{1,2}$};
\node [scale=1.15] at (-0.5,-3.65) {$\ell_{1,1}$};
\node [scale=1.15] at (0.5,-3.65) {$\ell_{1,2}$};
\node [scale=1.15] at (1.55,-2.85) {$\ell_{1,\alpha}$};

\node [scale=1.15] at (2.92,-0.8) {$s_{2,1}$};
\node [scale=1.15] at (4,-0.8) {$s_{2,2}$};
\node [scale=1.15] at (3.5,-3.65) {$\ell_{2,1}$};
\node [scale=1.15] at (4.5,-3.65) {$\ell_{2,2}$};
\node [scale=1.15] at (5.55,-2.85) {$\ell_{2,\alpha}$};

\node [scale=1.15] at (8.92,-0.8) {$s_{x,1}$};
\node [scale=1.15] at (10,-0.8) {$s_{x,2}$};
\node [scale=1.15] at (9.5,-3.65) {$\ell_{x,1}$};
\node [scale=1.15] at (10.5,-3.65) {$\ell_{x,2}$};
\node [scale=1.15] at (11.55,-2.85) {$\ell_{x,\alpha}$};

\end{tikzpicture}
\caption{Graphs $G$ with $\gamma_3(G)=\dim_3(G)+1$.}\label{fig_ex1}
\end{figure}

Based on Proposition~\ref{dom_upper} and Remark~\ref{remark_kdom_kdim}, we have the following

\begin{question}
Can we characterize graphs $G$ satisfying $\gamma_k(G)=\dim_k(G)+1$? 
\end{question}

\begin{question}
Can we characterize graphs $G$ satisfying $\gamma_k(G)=\dim_k(G)$? 
\end{question}

Next, we show that $\frac{\dim_k(G)}{\gamma_k(G)}$ can be arbitrarily large; thus, $\dim_k(G)-\gamma_k(G)$ can be arbitrarily large. We recall the connected graphs $G$ of order $n$ for which $\dim_k(G)\in\{1,n-2, n-1\}$; here, we note that Theorem~\ref{kdim_characterization}(a),(d) for the case $k=1$ is obtained in~\cite{adim}. See~\cite{broadcast} for a characterization of all graphs $G$ having $\dim_1(G)=m$ for each $m\in\mathbb{Z}^+$.\\

\begin{theorem}\label{kdim_characterization}
Let $G$ be a connected graph of order $n\ge2$, and let $k\in\mathbb{Z}^+$. Then $1\le\dim_k(G)\le n-1$, and we have the following:
\begin{itemize}
\item[(a)] \emph{\cite{moreno1}} $\dim_k(G)=1$ if and only if $G\in\cup_{i=2}^{k+2}\{P_i\}$;
\item[(b)] \emph{\cite{distKdim}} for $n\ge4$, $\dim_1(G)=n-2$ if and only if $G=P_4$, $G=K_{s,t}$ ($s,t\ge1$), $G=K_s+\overline{K}_t$ ($s\ge 1, t\ge2$), or $G=K_s+(K_1\cup K_t)$ ($s,t\ge 1$);
\item[(c)] \emph{\cite{distKdim}} for $k\ge 2$ and for $n\ge4$, $\dim_k(G)=n-2$ if and only if $G=K_{s,t}$ ($s,t\ge1$), $G=K_s+\overline{K}_t$ ($s\ge 1, t\ge2$), or $G=K_s+(K_1\cup K_t)$ ($s,t\ge 1$);
\item[(d)] \emph{\cite{distKdim}} $\dim_k(G)=n-1$ if and only if $G=K_n$.
\end{itemize}
\end{theorem}

\begin{proposition}
For a connected graph $G$ and for $k\in\mathbb{Z}^+$, $\frac{\dim_k(G)}{\gamma_k(G)}$ can be arbitrarily large. 
\end{proposition}

\begin{proof}
Let $G$ be a connected graph of order $n\ge4$. 
First, note that $\dim_k(K_n)=n-1$ by Theorem~\ref{kdim_characterization}(d) and $\gamma_k(K_n)=1$ by Observation~\ref{obs_dom}(a); thus $\frac{\dim_k(K_n)}{\gamma_k(K_n)}=n-1 \rightarrow \infty$ as $n\rightarrow \infty$. 

For another example, let $G$ be the graph obtained from $K_{1, \alpha}$, where $\alpha\ge 3$, by subdividing each edge of $K_{1, \alpha}$ exactly $k-1$ times; let $v$ be the central vertex of degree $\alpha$ in $G$ and let $\ell_1, \ell_2, \ldots, \ell_{\alpha}$ be the leaves of $G$ such that $d(v, \ell_i)=k$ for each $i\in[\alpha]$. Let $N(v)=\{s_1, s_2, \ldots, s_{\alpha}\}$ such that $s_i$ lies on the $v-\ell_i$ path in $G$, and let $P^i$ denote the $s_i-\ell_i$ path, where $i\in[\alpha]$. Then $\gamma_k(G)=1$ since $\{v\}$ is a minimum distance-$k$ dominating set of $G$ by Observation~\ref{obs_dom}(a). Note that $\dim_k(G)=\alpha-1$: (i) $\dim_k(G)\le \alpha-1$ since $N(v)-\{s_1\}$ forms a distance-$k$ resolving set of $G$; (ii) $\dim_k(G)\ge \alpha-1$ since $S \cap (V(P^i)\cup V(P^j))\neq\emptyset$ for any distance-$k$ resolving set $S$ of $G$ and for distinct $i,j\in[\alpha]$, as $S \cap (V(P^i)\cup V(P^j))=\emptyset$ implies $\code_{S,k}(s_i)=\code_{S,k}(s_j)$. So, $\frac{\dim_k(G)}{\gamma_k(G)}=\alpha-1 \rightarrow \infty$ as $\alpha\rightarrow \infty$.~\hfill
\end{proof}

Next, for any connected graph $G$ of order $n\ge2$ and for any $k\in\mathbb{Z}^+$, we show that $2 \le\gamma_k(G)+\dim_k(G)\le n$ and we characterize $G$ with $\gamma_k(G)+\dim_k(G)=2$. We recall the following results.

\begin{lemma}\emph{\cite{dimdom}}\label{no_twins}
Let $G$ be a connected graph. Then there exists a minimum dominating set for $G$ which does not have any pair of twin vertices.
\end{lemma}

\begin{theorem}\emph{\cite{dimdom}}\label{dimdom_characterization}
Let $G$ be a connected graph of order $n\ge 2$. Then $\gamma(G)+\dim(G) \le n$, and equality holds if and only if $G\in\{K_n, K_{s, n-s}\}$ for $2\le s \le n-2$.
\end{theorem}

\begin{proposition}\label{cor_dimdom}
Let $G$ be any connected graph of order $n\ge2$, and let $k\in\mathbb{Z}^+$. Then $2 \le \gamma_k(G)+\dim_k(G)\le n$, and $\gamma_k(G)+\dim_k(G)=2$ if and only $G\in\cup_{i=2}^{k+2}\{P_i\}$.
\end{proposition}

\begin{proof}
Let $G$ be a connected graph of order $n\ge2$, and let $k\in\mathbb{Z}^+$. Since $\gamma_k(G)\ge 1$ and $\dim_k(G)\ge 1$, we have $\gamma_k(G)+\dim_k(G)\ge2$. Note that $\gamma_k(G)+\dim_k(G)=2$ if and only if $\gamma_k(G)=1=\dim_k(G)$ if and only if  $G\in \cup_{i=2}^{k+2}\{P_i\}$ by Observation~\ref{obs_dom}(a) and Theorem~\ref{kdim_characterization}(a). 

To prove $\gamma_k(G)+\dim_k(G)\le n$, it suffices to show that $\gamma_1(G)+\dim_1(G)\le n$ by Observation~\ref{obs_mixed}. The proof given for Theorem~\ref{dimdom_characterization} in~\cite{dimdom} actually shows $\gamma_1(G)+\dim_1(G)\le n$. To see this, we can take a minimum dominating set $D$ of $G$ that contains no twin vertices by Lemma~\ref{no_twins}. Suppose $x,y\in D$ have the same neighbors in $V(G)-D$; this implies that neither $x$ nor $y$ has a neighbor in $D$, because if, say, $y$ has a neighbor in $D$, then $D-\{y\}$ remains a dominating set, and thus $x$ and $y$ have the same neighbors in $V(G)$, contradicting the choice of $D$. Since no two vertices of $D$ have the same neighborhood in $S=V(G)-D$, $S$ is a distance-$1$ resolving set of $G$, and we have $\gamma_1(G)+\dim_1(G)\le |D|+|S|=n$.~\hfill
\end{proof}

In contrast to Theorem~\ref{dimdom_characterization}, we note that if $G\in\{P_4, K_n, K_{s, n-s}\}$ with $2\le s \le n-2$, then $\gamma_1(G)+\dim_1(G)=|V(G)|$. So, we have the following

\begin{question}
Can we characterize graphs $G$ satisfying $\gamma_k(G)+\dim_k(G)=|V(G)|$? 
\end{question}


\section{Bounds on $\gamma_L^k(G)$}\label{sec_distK_loc_dom}

In this section, for any connected graph $G$ of order $n\ge2$ and for any $k\in\mathbb{Z}^+$, we show that $1\le \max\{\gamma_k(G), \dim_k(G)\} \le \gamma_L^k(G) \le\min\{\dim_k(G)+1, n-1\}$; we characterize $G$ satisfying $\gamma_L^k(G)=1$ and $\gamma_L^k(G)=n-1$, respectively. For any non-trivial tree $T$, we show that $\gamma_L^k(T) \le |V(T)|-ex(T)$ and we characterize trees $T$ achieving equality.

\begin{theorem}\label{b_main}
For any connected graph $G$ of order $n\ge2$ and for any $k\in\mathbb{Z}^+$, 
$$\max\{\gamma_k(G), \dim_k(G)\} \le \gamma_L^k(G) \le \min\{1+\dim_k(G), n-1\}.$$
\end{theorem}

\begin{proof}
Let $G$ be a connected graph of order $n\ge2$, and let $k\in\mathbb{Z}^+$. Since a minimum distance-$k$ locating-dominating set of $G$ is both a distance-$k$ dominating set of $G$ and a distance-$k$ resolving set of $G$, we have $\gamma_L^k(G)\ge \max\{\gamma_k(G), \dim_k(G)\}$. 

Next, we show that $\gamma_L^k(G)\le \min\{1+\dim_k(G), n-1\}$. Suppose $S$ is a minimum distance-$k$ resolving set of $G$; then at most one vertex in $G$ has the distance-$k$ metric code $(\textbf{k+1})_{|S|}$ with respect to $S$. If $\code_{S,k}(u)\neq (\textbf{k+1})_{|S|}$ for each $u\in V(G)$, then $S$ is a distance-$k$ locating-dominating set of $G$. If $\code_{S,k}(w)=(\textbf{k+1})_{|S|}$ for some $w\in V(G)$, then $S\cup\{w\}$ forms a distance-$k$ locating-dominating set of $G$. So, $\gamma_L^k(G)\le |S|+1=\dim_k(G)+1$. Now, $\gamma_L^k(G)\le n-1$ follows from the fact that any vertex subset $S' \subseteq V(G)$ with $|S'|=n-1$ is a distance-$k$ locating-dominating set of $G$.~\hfill  
\end{proof}

Theorems~\ref{kdim_characterization}(d) and~\ref{b_main} imply that, for $n\ge 2$ and for $k\ge1$, $\max\{\gamma_k(K_n), \dim_k(K_n)\}=\gamma_L^k(K_n)=\min\{1+\dim_k(K_n), n-1\}$. Since $\gamma_k(G)\ge 1$ and $\dim_k(G)\ge 1$, Theorem~\ref{b_main} implies the following.

\begin{corollary}
For any connected graph $G$ of order $n\ge2$ and for any $k\in\mathbb{Z}^+$, $1\le \gamma_L^k(G)\le n-1$. 
\end{corollary}

Next, we characterize connected graphs $G$ of order $n$ satisfying $\gamma_L^k(G)=1$ and $\gamma_L^k(G)=n-1$, respectively, for all $k\in\mathbb{Z}^+$. We recall the following observation.

\begin{observation}\emph{\cite{distKdim}}\label{obs_kdim}
Let $x$ and $y$ be distinct twin vertices of $G$, and let $k\in\mathbb{Z}^+$. Then, for any distance-$k$ resolving set $S_k$ of $G$, $S_k\cap\{x,y\}\neq\emptyset$.
\end{observation}

\begin{theorem}\label{loc_dom_characterization}
Let $G$ be a connected graph of order $n\ge 2$, and let $k\in\mathbb{Z}^+$. Then
\begin{itemize}
\item[(a)] $\gamma_L^k(G)=1$ if and only if $G\in\cup_{i=2}^{k+1}\{P_i\}$;
\item[(b)] $\gamma_L^1(G)=n-1$ if and only if $G\in\{K_n, K_{1, n-1}\}$;
\item[(c)] for $k\ge 2$, $\gamma_L^k(G)=n-1$ if and only if $G=K_n$.
\end{itemize}
\end{theorem}

\begin{proof}
Let $G$ be a connected graph of order $n\ge 2$, and let $k\in\mathbb{Z}^+$. 

(a) If $G\in \cup_{i=2}^{k+1}\{P_i\}$, then a leaf of $G$ forms a distance-$k$ locating-dominating set of $G$; thus, $\gamma_L^k(G)=1$. Now, suppose $\gamma_L^k(G)=1$; then $\gamma_k(G)=1=\dim_k(G)$. By Theorem~\ref{kdim_characterization}(a), $\dim_k(G)=1$ implies $G\in \cup_{i=2}^{k+2}\{P_i\}$, where any minimum distance-$k$ resolving set consists of a leaf whereas a leaf of $P_{k+2}$ fails to form a distance-$k$ dominating set of $P_{k+2}$ since $\diam(P_{k+2})=k+1$. So, $\gamma_L^k(G)=1$ implies $G\in \cup_{i=2}^{k+1}\{P_i\}$.

(b) First, suppose $G\in\{K_n, K_{1, n-1}\}$. Note that $\gamma_L^1(K_n)=n-1$ by Theorems~\ref{kdim_characterization}(d) and~\ref{b_main}. For $n \ge 3$, if $v$ is the central vertex of $K_{1, n-1}$ and $N(v)=\{s_1, s_2, \ldots, s_{n-1}\}$, then $\dim_1(K_{1, n-1})=n-2$ by Theorem~\ref{kdim_characterization}(b) and $|S\cap N(v)|=n-2$ for any minimum distance-$1$ resolving set $S$ of $K_{1, n-1}$ by Observation~\ref{obs_kdim}; without loss of generality, let $S'=\{s_1, s_2, \ldots, s_{n-2}\}$ be a minimum distance-$1$ resolving set of $K_{1, n-1}$. Since $d(s_{n-1}, S')=2$, $S'$ fails to be a distance-$1$ locating-dominating set of $K_{1, n-1}$; thus, $\gamma_L^1(K_{1, n-1})\ge n-1$. By Theorem~\ref{b_main}, $\gamma_L^1(K_{1, n-1})=n-1$. 

Second, suppose $\gamma_L^1(G)=n-1$. By Theorem~\ref{b_main}, $\dim_1(G)\in\{n-2, n-1\}$. To see this, if $\dim_1(G) \le n-3$, then $\gamma_L^1(G)\le \dim_1(G)+1\le n-2$ by Theorem~\ref{b_main}. If $\dim_1(G)=n-1$, then $G=K_n$ by Theorem~\ref{kdim_characterization}(d). If $\dim_1(G)=n-2$, then $G=P_4$, $G=K_{s,t}$ with $s,t\ge1$, $G=K_s+\overline{K}_t$ with $s\ge 1, t\ge2$, or $G=K_s+(K_1\cup K_t)$ with $s,t\ge 1$ by Theorem~\ref{kdim_characterization}(b). We note the following: (i) $\gamma_L^1(P_4)=2$ since the two leaves of $P_4$ form a minimum distance-$1$ locating-dominating set of $P_4$; (ii) $\gamma_L^1(K_{1,t})=\gamma_L^1(K_1+\overline{K}_t)=t$ as shown above; (iii) for $s,t\ge 2$, $\gamma_L^1(K_{s,t})=s+t-2=\gamma_L^1(K_{s}+\overline{K}_t)$ since all but one vertex from each of the two partite sets form a minimum distance-$1$ locating-dominating set of $K_{s,t}$; (iv) $K_1+(K_1 \cup K_1)=K_{1,2}$ and $\gamma_L^1(K_{1,2})=2$ as shown above; (v) for $t\ge2$, $\gamma_L^1(K_1+(K_1 \cup K_t))=t$ since all but one vertex of the $K_t$ and the leaf of $K_1+(K_1 \cup K_t)$ form a minimum distance-$1$ locating-dominating set of $K_1+(K_1 \cup K_t)$; (vi) for $s\ge 2$ and $t\ge1$, $\gamma_L^1(K_s+(K_1 \cup K_t))=s+t-1$ since all but one vertex of the $K_s$ and all vertices of the $K_t$ form a minimum distance-$1$ locating-dominating set of $K_s+(K_1 \cup K_t)$. So, $\gamma_L^1(G)=n-1$ implies $G=K_n$ or $G=K_{1, n-1}$.

(c) Let $k\ge 2$. Note that $\gamma_L^k(K_n)=n-1$ by Theorems~\ref{kdim_characterization}(d) and~\ref{b_main}. So, suppose $\gamma_L^k(G)=n-1$. Then $\dim_k(G)\in\{n-1,n-2\}$ by Theorem~\ref{b_main}. If $\dim_k(G)=n-1$, then $G=K_n$ by Theorem~\ref{kdim_characterization}(d). If $\dim_k(G)=n-2$ for $n\ge4$, then, by Theorem~\ref{kdim_characterization}(c), $G=K_{s,t}$ with $s,t\ge1$, $G=K_s+\overline{K}_t$ with $s\ge 1, t\ge2$, or $G=K_s+(K_1\cup K_t)$ with $s,t\ge 1$; then $\diam(G)=2$ and any minimum distance-$k$ resolving set of $G$ is also a distance-$k$ dominating set of $G$. So, $\dim_k(G)=n-2$ implies $\gamma_L^k(G)=n-2$ for $k\ge 2$.~\hfill
\end{proof}

\begin{question}
Can we characterize graphs $G$ of order $n$ such that $\gamma_L^k(G)=\beta$, where $\beta\in\{2,3,\ldots, n-2\}$? 
\end{question} 

Next, we examine the relation between $\gamma_L^k(G)$ and other parameters in Theorem~\ref{b_main}.

\begin{proposition}\label{parameter_gap} 
Let $G$ be a non-trivial connected graph, and let $k\in\mathbb{Z}^+$. Then 
\begin{itemize}
\item[(a)] $\gamma_L^k(G)-\dim_k(G) \in\{0,1\}$; 
\item[(b)] $\gamma_L^k(G)-\gamma_k(G)$ can be arbitrarily large;
\item[(c)] $(|V(G)|-1)-\gamma_L^k(G)$ can be arbitrarily large.
\end{itemize}
\end{proposition}

\begin{proof}
Let $k\in\mathbb{Z}^+$. For (a), $0\le \gamma_L^k(G)-\dim_k(G)\le 1$ by Theorem~\ref{b_main}.

For (b) and (c), let $G$ be a tree obtained from the path $v_1, v_2, \ldots, v_x$ ($x\ge2$) by adding leaves $\ell_{i,1}, \ell_{i,2}, \ldots, \ell_{i, \alpha}$ ($\alpha\ge3$) to each vertex $v_i$, where $i\in[x]$; notice that $|V(G)|=x(\alpha+1)$. Since $\cup_{i=1}^{x}\{v_i\}$ is a distance-$k$ dominating set of $G$, $\gamma_k(G)\le x$. Note that $\gamma_L^k(G) \ge x(\alpha-1)$ by Observation~\ref{obs_kdim} since any distinct vertices in $\{\ell_{i,1}, \ell_{i,2}, \ldots, \ell_{i, \alpha}\}$ are twin vertices in $G$. Also, note that $\gamma_L^k(G) \le x\alpha$ since $V(G)-\cup_{i=1}^{x}\{\ell_{i, \alpha}\}$ is a distance-$k$ locating-dominating set of $G$. So, $\gamma_L^k(G)-\gamma_k(G)\ge x(\alpha-1)-x=x(\alpha-2) \rightarrow \infty$ as $x \rightarrow \infty$ or $\alpha\rightarrow\infty$, and $|V(G)|-1-\gamma_L^k(G)\ge x(\alpha+1)-1-x\alpha=x-1 \rightarrow\infty$ as $x \rightarrow \infty$.~\hfill
\end{proof}

In view of Theorem~\ref{b_main} and Proposition~\ref{parameter_gap}(b), we have the following

\begin{question}
Can we characterize graphs $G$ such that $\gamma_L^k(G)=\gamma_k(G)$?
\end{question}

Next, for a graph $G$ with $\gamma_L^k(G)=\beta$, we determine the upper bound of $|V(G)|$.

\begin{theorem}\emph{\cite{distKdim}}\label{kdim_maxorder}
If $\dim_k(G)=\beta$, then $|V(G)|\le (\lfloor\frac{2(k+1)}{3}\rfloor+1)^{\beta}+\beta\sum_{i=1}^{\lceil\frac{k+1}{3}\rceil}(2i-1)^{\beta-1}$.
\end{theorem}

By Theorem~\ref{b_main}, $\gamma_L^k(G)=\beta$ implies $\dim_k(G)\le \beta$.  Theorem~\ref{kdim_maxorder} is sharp, and a graph $G$ attaining the maximum order must contain a vertex $\omega\in V(G)$ with $\code_{S,k}(\omega)=(\textbf{k+1})_{|S|}$ for any minimum distance-$k$ resolving set $S$ of $G$. The deletion of $\omega$ from $G$ leaves intact distance relations and code vectors; thus, we have the following sharp bound.

\begin{corollary}\label{L_maxorder}
If $\gamma_L^k(G)=\beta$, then $|V(G)|\le(\lfloor\frac{2(k+1)}{3}\rfloor+1)^{\beta}-1+\beta\sum_{i=1}^{\lceil\frac{k+1}{3}\rceil}(2i-1)^{\beta-1}$.
\end{corollary}

\begin{remark}
The proof for Theorem~\ref{kdim_maxorder} in~\cite{distKdim} uses a method similar to the one in~\cite{Hernando}. For a construction of graphs $G$ with $\dim_1(G)=\beta$ of maximum order $\beta+2^{\beta}$, we refer to~\cite{broadcast}. For a construction of graphs $G$ with $\dim_2(G)=\beta$ and of order $\beta+3^{\beta}$, we refer to~\cite{distKdim}; this construction is similar to the one provided in~\cite{Jesse}. 
\end{remark}

Next, for any non-trivial tree $T$ and for $k\in\mathbb{Z}^+$, we show that $\gamma_L^k(T) \le n-ex(T)$ and we characterize trees $T$ achieving equality.

\begin{proposition}\label{b_tree_upperbound}
For any tree $T$ of order $n\ge2$ and for any $k\in\mathbb{Z}^+$, $\gamma_L^k(T) \le n-ex(T)$. 
\end{proposition}

\begin{proof}
Let $T$ be a tree of order $n\ge 2$ and let $k\in\mathbb{Z}^+$. If $ex(T)\in\{0,1\}$, then $\gamma_L^k(T)\le n-1\le n-ex(T)$ by Theorem~\ref{b_main}. So, suppose $ex(T)=x\ge 2$; let $v_1, v_2, \ldots, v_x$ be the exterior major vertices of $T$. For each $i\in[x]$, let $\{\ell_{i, 1}, \ell_{i,2}, \ldots, \ell_{i, \sigma_i}\}$ be the set of terminal vertices of $v_i$ in $T$ with $ter(v_i)=\sigma_i\ge 1$. Since $S=V(T)-\cup_{i=1}^{x}\{\ell_{i,1}\}$ is a distance-$k$ locating-dominating set of $T$ with $|S|=n-x=n-ex(T)$, $\gamma_L^k(T) \le n-ex(T)$.~\hfill
\end{proof}

Next, we characterize non-trivial trees $T$ satisfying $\gamma_L^k(T)= |V(T)|-ex(T)$. We recall some terminology. An \emph{exterior degree-two vertex} is a vertex of degree two that lies on a path from a terminal vertex to its major vertex, and an \emph{interior degree-two vertex} is a vertex of degree two such that the shortest path to any terminal vertex includes a major vertex.

\begin{theorem}\label{b_tree_equality}
Let $T$ be any tree of order $n\ge 2$ and let $k\in\mathbb{Z}^+$. Then 
$\gamma_L^k(T)=n-ex(T)$ if and only if $k=1$, $ex(T)\ge 1$, and $ex(T)+\sigma(T)=n$.
\end{theorem}

\begin{proof}
Let $T$ be a tree of order $n\ge 2$ and let $k\in\mathbb{Z}^+$. If $ex(T)=x\ge 1$, let $v_1,v_2, \ldots, v_x$ be the exterior major vertices of $T$, and let $\{\ell_{i,1}, \ell_{i,2}, \ldots, \ell_{i, \sigma_i}\}$ be the set of terminal vertices of $v_i$ with $ter(v_i)=\sigma_i\ge1$ in $T$ for each $i\in[x]$. 

($\Leftarrow$) Let $k=1$, $ex(T)=x\ge 1$, and $ex(T)+\sigma(T)=n$; notice that $T$ is a caterpillar. Let $S$ be an arbitrary minimum distance-$1$ locating-dominating set of $T$. By Observation~\ref{obs_kdim}, $|S \cap \{\ell_{i,1}, \ell_{i,2}, \ldots, \ell_{i, \sigma_i}\}|\ge \sigma_i-1$. Thus, up to a relabeling of vertices of $T$, we may assume that $S\supseteq V(T)-\cup_{i=1}^{x}\{v_i, \ell_{i,1}\} $. Since $N[\ell_{i,1}] \cap N[\ell_{j,1}]=\emptyset$ for $i\neq j$, a vertex in $\{v_i, \ell_{i,1}\}$ (for each $i\in[x]$) must also belong to $S$ by Observation~\ref{obs_dom}(b). So, $\gamma_L^1(T)\ge n-ex(T)$. Since $\gamma_L^1(T) \le n-ex(T)$ by Proposition~\ref{b_tree_upperbound}, $\gamma_L^1(T)=n-ex(T)$.

($\Rightarrow$) Let $\gamma_L^k(T)=n-ex(T)$. If $ex(T)=0$, then $\gamma_L^k(T)<n-ex(T)$ by Theorem~\ref{b_main}. So, let $ex(T)=x\ge 1$. We will show that $T$ has no major vertex of terminal degree zero and no degree-two vertex; i.e., each vertex in $T$ is either an exterior major vertex or a leaf.

If $T$ contains either an interior degree-two vertex $w$ or a major vertex $w'$ with $ter(w')=0$, then $A=V(T)-(\{u\} \cup (\cup_{i=1}^{x}\{\ell_{i,1}\}))$, where $u\in\{w, w'\}$, forms a distance-$k$ locating-dominating set of $T$; thus $\gamma_L^k(T)\le n-(x+1)<n-ex(T)$. Now, suppose $T$ contains an exterior degree-two vertex, say $z$. By relabeling the vertices of $T$ if necessary, we may assume that $z$ lies on the $v_i-\ell_{i,1}$ path in $T$ for some $i\in[x]$. If $ter(v_i)\ge 2$, then $B=V(T)-(\{z\} \cup (\cup_{j=1}^{x}\{\ell_{j, \sigma_j}\}))$ forms a distance-$k$ locating-dominating set of $T$. If $ter(v_i)=1$, then $C=V(T)-(\{v_i\} \cup (\cup_{j=1}^{x}\{\ell_{j,1}\}))$ forms a distance-$k$ locating-dominating set of $T$. (It is easy to see that the sets $A$, $B$, and $C$ are distance-$1$ locating-dominating; then apply  Observation~\ref{obs_mixed}(c) for $k\geq 1$.) In each case, $\gamma_L^k(T)\le n-(x+1)<n-ex(T)$. 

So, each vertex in $T$ is either an exterior major vertex or a leaf; thus $ex(T)+\sigma(T)=n$. Now, if $k\ge 2$, then $R=V(T)-(\{v_1\} \cup (\cup_{i=1}^{x} \{\ell_{i,1}\}))$ forms a distance-$k$ locating-dominating set of $T$, and hence $\gamma_L^k(T)\le |R|=n-ex(T)-1<n-ex(T)$. Thus, $k=1$.~\hfill
\end{proof}


\section{$\gamma_L^k(G)$ of some classes of graphs}\label{sec_graphs}

In this section, for any $k\in\mathbb{Z}^+$, we determine $\gamma_L^k(G)$ when $G$ is the Petersen graph, a complete multipartite graph, a cycle or a path. We begin with the following observations.

\begin{observation}\emph{\cite{moreno1, distKdim}}\label{obs_kdim_d}
Let $G$ be a connected graph with $\diam(G)=d\ge2$, and let $k\in\mathbb{Z}^+$. If $k\ge d-1$, then $\dim_k(G)=\dim(G)$.
\end{observation}

\begin{observation}\label{obs_mixed2}
Let $G$ be any connected graph, and let $k, k'\in\mathbb{Z}^+$. Then 
\begin{itemize}
\item[(a)] for $k>k'$, $\gamma_L^k(G) \le \gamma_L^{k'}(G)\le \gamma_L^1(G)$;
\item[(b)] if $k\ge \diam(G)$, then $\gamma_L^k(G)=\dim_k(G)$.
\end{itemize}
\end{observation}

Next, we determine $\gamma_L^k(\mathcal{P})$ for the Petersen graph $\mathcal{P}$.

\begin{example} Let $\mathcal{P}$ be the Petersen graph with the the following presentation: two disjoint copies of $C_5$ are given by $u_1, u_2, u_3,u_4,u_5,u_1$ and $w_1,w_3,w_5,w_2,w_4,w_1$, respectively, and the remaining edges are $u_iw_i$ for each $i\in[5]$. Then, for $k\in \mathbb{Z}^+$, 
$$\gamma_L^k(\mathcal{P})=\left\{
\begin{array}{ll}
\dim_k(\mathcal{P})+1=4 & \mbox{ if } k=1,\\
\dim_k(\mathcal{P})=3 & \mbox{ if } k\ge 2.
\end{array}\right.$$
To see this, note that $\dim(\mathcal{P})=3$ (see~\cite{petersen}) and $\diam(\mathcal{P})=2$. For any $k\ge2$, $\gamma_L^k(\mathcal{P})=\dim_k(\mathcal{P})=\dim(\mathcal{P})=3$ by Observations~\ref{obs_kdim_d} and ~\ref{obs_mixed2}(b). Next, we show that $\gamma_L^1(\mathcal{P})=4$. For any minimum distance-$1$ resolving set $S$ of $\mathcal{P}$, we may assume $u_1\in S$ since $\mathcal{P}$ is vertex-transitive. It was shown in~\cite{mbrg} that there are six such $S$ containing $u_1$ (i.e., $\{u_1,w_2,w_3\}$, $\{u_1,u_4,w_2\}$, $\{u_1,w_4,w_5\}$, $\{u_1,u_3,w_5\}$, $\{u_1,u_4,w_3\}$ and $\{u_1,u_3,w_4\}$). Since none of those six sets $S$ containing $u_1$ form a distance-$1$ dominating set of $\mathcal{P}$, $\gamma_L^1(\mathcal{P})\ge \dim_1(\mathcal{P})+1=4$. Since $\{u_1, u_4, w_2, w_3\}$ is a distance-$1$ locating-dominating set of $\mathcal{P}$, $\gamma_L^1(\mathcal{P})\le 4$; thus, $\gamma_L^1(\mathcal{P})=\dim_1(\mathcal{P})+1=4$.
\end{example}

Next, we determine $\gamma_L^k(G)$ when $G$ is a complete multipartite graph. 

\begin{proposition}\emph{\cite{multipartite}}\label{dim_kpartite}
For $m \ge 2$, let $G=K_{a_1, a_2, \ldots, a_m}$ be a complete $m$-partite graph of order $n=\sum_{i=1}^{m}a_i\ge3$. Let $s$ be the number of partite sets of $G$ consisting of exactly one element. Then 
$$\dim(G)=\left\{
\begin{array}{ll}
n-m & \mbox{ if } s=0,\\
n-m+s-1 & \mbox{ if } s \neq 0.
\end{array}\right.$$
\end{proposition}

\begin{proposition}
For $m \ge 2$, let $G=K_{a_1, a_2, \ldots, a_m}$ be a complete $m$-partite graph of order $n=\sum_{i=1}^{m}a_i\ge3$. For $k\in\mathbb{Z}^+$,  
$$\gamma_L^k(G)=\left\{
\begin{array}{ll}
\dim_k(G)+1=n-1 & \mbox{ if } k=1 \mbox{ and } G=K_{1, n-1},\\
\dim_k(G) & \mbox{ otherwise}.
\end{array}\right.$$
\end{proposition}

\begin{proof}
Let $G=K_{a_1, a_2, \ldots, a_m}$ be a complete $m$-partite graph of order $n=\sum_{i=1}^{m}a_i\ge3$, where $m\ge 2$, and let $k\in\mathbb{Z}^+$. Note that $\diam(G)\in\{1,2\}$, where $\diam(G)=1$ if and only if $G=K_n$ and $\gamma_L^k(K_n)=\dim_k(K_n)=n-1$, for any $k\ge1$, by Theorems~\ref{kdim_characterization}(d) and~\ref{b_main}. If $\diam(G)=2$ and $k\ge 2$, then $\gamma_L^k(G)=\dim_k(G)=\dim(G)$ by Observations~\ref{obs_kdim_d} and~\ref{obs_mixed2}(b). So, suppose $\diam(G)=2$ and $k=1$. Let $s$ be the number of partite sets of $G$ consisting of exactly one element. If $s=0$, then any minimum distance-$1$ resolving set of $G$ is also a distance-$1$ dominating set of $G$; thus, $\gamma_L^1(G)=\dim_1(G)$. If $s=1$ with $m=2$, then $G=K_{1, n-1}$ and $\gamma_L^1(K_{1,n-1})=n-1=\dim_1(K_{1, n-1})+1$ by Theorems~\ref{kdim_characterization}(b) and~\ref{loc_dom_characterization}(b). If either $s=1$ with $m\ge 3$ or $s\ge 2$, then any minimum distance-$1$ resolving set of $G$ is also a distance-$1$ dominating set of $G$, and hence $\gamma_L^1(G)=\dim_1(G)$.~\hfill
\end{proof}

Next, we determine $\gamma_L^k(G)$ when $G$ is a cycle or a path. 

\begin{theorem}\emph{\cite{distKdim}}\label{kdim_pathcycle}
Let $k\in\mathbb{Z}^+$. Then 
\begin{itemize}
\item[(a)] $\dim_k(P_n)=1$ for $2\le n \le k+2$;
\item[(b)] $\dim_k(C_n)=2$ for $3\le n \le 3k+3$, and $\dim_k(P_n)=2$ for $k+3 \le n \le 3k+3$;
\item[(c)] for $n \ge 3k+4$, 
$$\dim_{k}(C_n)=\dim_k(P_n)=\left\{
\begin{array}{ll}
\lfloor\frac{2n+3k-1}{3k+2}\rfloor & \mbox{ if } n \equiv 0,1,\ldots, k+2 \!\!\pmod{(3k+2)},\\ 
\lfloor\frac{2n+4k-1}{3k+2}\rfloor & \mbox{ if } n \equiv k+3,\ldots,  \lceil \frac{3k+5}{2}\rceil-1 \!\!\pmod{(3k+2)},\\ 
\lfloor\frac{2n+3k-1}{3k+2}\rfloor & \mbox{ if } n \equiv \lceil \frac{3k+5}{2}\rceil,\ldots, 3k+1 \!\!\pmod{(3k+2)}.
\end{array}\right.$$
\end{itemize}
\end{theorem}

\begin{proposition}
Let $G=P_n$ for $n\ge 2$ or $G=C_n$ for $n\ge 3$. For any $k\in\mathbb{Z}^+$, 
$$\gamma_L^k(G)=\left\{
\begin{array}{ll}
\dim_k(G)+1 & \mbox{if } G\in\{P_n, C_n\} \mbox{ and } n\equiv1 \!\!\pmod{(3k+2)},\\ 
{} & \mbox{or } G=P_n \mbox{ and } n\equiv k+2 \!\!\pmod{(3k+2)},\\
{} & \mbox{or } G=C_n, n\ge3k+4, \mbox{ and } n\equiv k+2 \!\!\pmod{(3k+2)},\\
\dim_k(G) & \mbox{otherwise}.
\end{array}\right.$$
\end{proposition}

\begin{proof}
Let $G=P_n$ for $n\ge 2$ or $G=C_n$ for $n\ge 3$. Let $k\in\mathbb{Z}^+$. 

If $2\le n \le k+1$, then $\gamma_L^k(P_n)=\dim_k(P_n)=1$ by Theorems~\ref{loc_dom_characterization}(a) and~\ref{kdim_pathcycle}(a). If $n=k+2$, then $\gamma_L^k(P_{k+2})=\dim_k(P_{k+2})+1=2$ by Theorems~\ref{b_main},~\ref{loc_dom_characterization}(a) and~\ref{kdim_pathcycle}(a). If $k+3 \le n \le 3k+2$ and $P_n$ is obtained from $C_n$, given by $u_0, u_1, \ldots, u_{n-1}, u_0$, by deleting the edge $u_ku_{k+1}$, then $\{u_0, u_{\alpha}\}$, where $\alpha=\min\{2k+1,n-1\}$, forms a distance-$k$ locating-dominating set of $P_n$, and thus $\gamma_L^k(P_n)=\dim_k(P_n)=2$ by Theorems~\ref{b_main} and~\ref{kdim_pathcycle}(b). If $3\le n\le 3k+2$ and $C_n$ is given by $u_0, u_1, \ldots, u_{n-1}, u_0$, then $\{u_0, u_{\alpha}\}$, where $\alpha=\min\{2k+1,n-1\}$, forms a distance-$k$ locating-dominating set of $C_n$, and thus $\gamma_L^k(C_n)=\dim_k(C_n)=2$ Theorems~\ref{b_main} and~\ref{kdim_pathcycle}(b). If $n=3k+3$, then, for any minimum distance-$k$ resolving set $R$ of $G\in\{P_{3k+3}, C_{3k+3}\}$, there is a vertex $w$ in $G$ with $\code_{R,k}(w)=(k+1,k+1)$; thus, $\gamma_L^k(G)=\dim_k(G)+1=3$ by Theorem~\ref{b_main}. 

Now, suppose $n\ge 3k+4$, and let $G\in\{P_n, C_n\}$; then $\dim_k(G)\ge 3$. Let $S$ be any minimum distance-$k$ resolving set of $G$. First, suppose that $|S|$ is odd. If $n\not\equiv k+2 \pmod{(3k+2)}$, then there exists a minimum distance-$k$ resolving set $S_0$ of $G$ such that $S_0$ is also a distance-$k$ dominating set of $G$ (see~\cite{distKdim}); thus, $\gamma_L^k(G)=\dim_k(G)$. If $n \equiv k+2 \pmod{(3k+2)}$, then there exists a vertex $w$ in $G$ with $\code_{R,k}(w)=(\textbf{k+1})_{|R|}$ for any minimum distance-$k$ resolving set $R$ of $G$ (see~\cite{distKdim}); thus, $\gamma_L^k(G)=\dim_k(G)+1$. Second, suppose $|S|$ is even. If $n \not\equiv 1 \pmod{(3k+2)}$, then there exists a minimum distance-$k$ resolving set $S_1$ of $G$ such that $S_1$ is also a distance-$k$ dominating set of $G$ (see~\cite{distKdim}); thus, $\gamma_L^k(G)=\dim_k(G)$. If $n \equiv 1 \pmod{(3k+2)}$, then there exists a vertex $w$ in $G$ with $\code_{S,k}(w)=(\textbf{k+1})_{|S|}$ for any minimum distance-$k$ resolving set $S$ of $G$ (see~\cite{distKdim}); thus, $\gamma_L^k(G)=\dim_k(G)+1$.~\hfill
\end{proof}

Based on the proof of Theorem~\ref{b_main}, we note that $\gamma_L^k(G)= \dim_k(G)+1$ if and only if, for every minimum distance-$k$ resolving set $S$ of $G$, there exists a vertex $w\in V(G)-S$ with $d(w, S)>k$. In other words, if there exists a minimum distance-$k$ resolving set $S'$ of $G$ such that $d(v, S') \le k$ for each $v\in V(G)$, then $\gamma_L^k(G)=\dim_k(G)$. 

\begin{question}
Since $\dim_k(G)\leq \gamma_L^k(G)\leq\dim_k(G)+1$, can we characterize $G$ for which each of the two (end) inequalities is an equality?
\end{question}


\section{The effect of edge deletion on $\gamma_L^k(G)$}\label{sec_edge}

In this section, we examine the effect of edge deletion on the distance-$k$ location-domination number of graphs. Throughout the section, let both $G$ and $G-e$, where $e\in E(G)$, be connected graphs. For the effect of edge deletion on the metric dimension of graphs, we refer to~\cite{joc}. We recall how the distance-$k$ dimension of a graph changes upon deletion of an edge. 

\begin{theorem}\label{T+e}
Let $G$ be a connected graph with $e \in E(G)$, and let $k\in\mathbb{Z}^+$. Then  
\begin{itemize}
\item[(a)] \emph{\cite{broadcast, juan_again}} $\dim_1(G)-1 \le \dim_1(G-e) \le \dim_1(G)+1$;
\item[(b)] \emph{\cite{distKdim}} $\dim_2(G-e) \le \dim_2(G)+1$; 
\item[(c)] \emph{\cite{distKdim}} for $k\ge3$, $\dim_k(G-e)\le \dim_k(G)+2$;
\item[(d)] \emph{\cite{distKdim}} for $k\ge 2$, $\dim_k(G)-\dim_k(G-e)$ can be arbitrarily large.
\end{itemize}
\end{theorem}

\begin{theorem}
Let $G$ be a connected graph with $e\in E(G)$, and let $k\in\mathbb{Z}^+$. Then
\begin{itemize}
\item[(a)] $\gamma_L^1(G)-2\le \gamma_L^1(G-e) \le \gamma_L^1(G)+2$;
\item[(b)] $\gamma_L^2(G-e) \le \gamma_L^2(G)+2$;
\item[(c)] for $k\ge 3$,  $\gamma_L^k(G-e) \le \gamma_L^k(G)+3$.
\end{itemize}
\end{theorem}

\begin{proof}
Let $k\in\mathbb{Z}^+$. By Theorem~\ref{b_main}, we have $\dim_k(G) \le \gamma_L^k(G)\le \dim_k(G)+1$ and $\dim_k(G-e) \le \gamma_L^k(G-e)\le \dim_k(G-e)+1$.

For (a), note that $\gamma_L^1(G-e)-\gamma_L^1(G)\ge \dim_1(G-e)-(\dim_1(G)+1)\ge -2$ and $\gamma_L^1(G)-\gamma_L^1(G-e) \ge \dim_1(G)-(\dim_1(G-e)+1) \ge -2$ by Theorem~\ref{T+e}(a); thus, $\gamma_L^1(G)-2\le\gamma_L^1(G-e)\le\gamma_L^1(G) + 2$.

For (b), note that $\gamma_L^2(G)-\gamma_L^2(G-e)\ge \dim_2(G)-(\dim_2(G-e)+1)\ge -2$ by Theorem~\ref{T+e}(b); thus $\gamma_L^2(G-e)\le\gamma_L^2(G)+2$.

For (c), for any $k\ge 3$, we have $\gamma_L^k(G)-\gamma_L^k(G-e)\ge \dim_k(G)-(\dim_k(G-e)+1)\ge -3$ by Theorem~\ref{T+e}(c); thus $\gamma_L^k(G-e)\le\gamma_L^k(G)+3$.~\hfill
\end{proof}

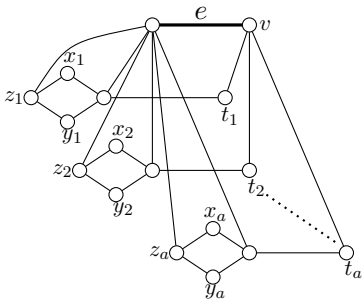
\begin{figure}[!ht]
\centering
\begin{tikzpicture}[scale=.645, transform shape]

\node [draw, shape=circle, scale=.8] (1) at  (3, 1.5) {};
\node [draw, shape=circle, scale=.8] (2) at  (5, 1.5) {};
\node [draw, shape=circle, scale=.8] (3) at  (2, 0) {};
\node [draw, shape=circle, scale=.8] (4) at  (4.5, 0) {};
\node [draw, shape=circle, scale=.8] (5) at  (3,-1.5) {};
\node [draw, shape=circle, scale=.8] (6) at  (5, -1.5) {};
\node [draw, shape=circle, scale=.8] (7) at  (5, -3.2) {};
\node [draw, shape=circle, scale=.8] (8) at  (7, -3.2) {};

\node [draw, shape=circle, scale=.8] (a1) at  (0.5, 0) {};
\node [draw, shape=circle, scale=.8] (a2) at  (1.25, 0.5) {};
\node [draw, shape=circle, scale=.8] (a3) at  (1.25, -0.5) {};

\node [draw, shape=circle, scale=.8] (b1) at  (1.5, -1.5) {};
\node [draw, shape=circle, scale=.8] (b2) at  (2.25, -1) {};
\node [draw, shape=circle, scale=.8] (b3) at  (2.25, -2) {};

\node [draw, shape=circle, scale=.8] (c1) at  (3.5, -3.2) {};
\node [draw, shape=circle, scale=.8] (c2) at  (4.25, -2.7) {};
\node [draw, shape=circle, scale=.8] (c3) at  (4.25, -3.7) {};

\node [scale=1.2] at (1.4,0.8) {$x_1$};
\node [scale=1.2] at (1.35,-0.8) {$y_1$};
\node [scale=1.2] at (2.4,-0.7) {$x_2$};
\node [scale=1.2] at (2.4,-2.3) {$y_2$};
\node [scale=1.2] at (4.3,-2.4) {$x_a$};
\node [scale=1.2] at (4.3,-4) {$y_a$};
\node [scale=1.2] at (4.6,-0.38) {$t_1$};
\node [scale=1.2] at (5.15,-1.85) {$t_2$};
\node [scale=1.2] at (7.15,-3.55) {$t_a$};
\node [scale=1.2] at (5.33,1.5) {$v$};
\node [scale=1.2] at (0.15,0) {$z_1$};
\node [scale=1.2] at (1.15,-1.5) {$z_2$};
\node [scale=1.2] at (3.15,-3.2) {$z_a$};

\node [scale=1.5] at (4,1.75) {\bf$e$};

\draw(1)--(3)--(4)--(2);\draw(1)--(5)--(6)--(2);\draw(1)--(7)--(8)--(2);
\draw(a1)--(a2)--(3)--(a3)--(a1);
\draw(a1) .. controls (1.25,1.15) .. (1);
\draw(b1)--(b2)--(5)--(b3)--(b1);\draw(b1)--(1);
\draw(c1)--(c2)--(7)--(c3)--(c1);\draw(c1)--(1);
\draw[very thick](1)--(2);

\draw[thick, dotted] (5.35,-2)--(6.8,-3);

\end{tikzpicture}
\caption{\small \cite{distKdim} Graphs $G$ such that $\dim_k(G) - \dim_k(G-e)$ can be arbitrarily large, where $k\ge 2$ and $a \ge 3$.}\label{fig_kdim_edge}
\end{figure}

\begin{theorem}
For any integer $k\ge2$, $\gamma_L^k(G)-\gamma_L^k(G-e)$ can be arbitrarily large.
\end{theorem}

\begin{proof}
Let $G$ be the graph in Fig.~\ref{fig_kdim_edge} with $a\ge 3$. It was shown in~\cite{distKdim} that, for any $k\ge 2$, $\dim_k(G)=2a$ and $\dim_k(G-e)=a+1$. For $k\ge 2$, $\gamma_L^k(G)\ge \dim_k(G)=2a$ and $\gamma_L^k(G-e) \le \dim_k(G-e)+1=a+2$ by Theorem~\ref{b_main}; thus, $\gamma_L^k(G)-\gamma_L^k(G-e) \ge 2a-(a+2)=a-2 \rightarrow\infty$ as $a\rightarrow \infty$.~\hfill
\end{proof}

\textbf{Acknowledgement.} We much appreciate the anonymous referees for their careful reading, the correction of an error, and helpful comments which improved the paper.


\end{document}